\documentclass[12pt,a4paper]{article}
\usepackage{amsfonts}
\usepackage{mathrsfs}
\usepackage{amsmath}
\usepackage{amssymb}
\usepackage{cite}
\usepackage{relsize}
\usepackage{mathtools}
\usepackage{amsthm}
\usepackage[pdftex]{graphicx}

\usepackage[small,nohug,heads=littlevee]{diagrams}

\DeclareMathOperator{\End}{End} 
 
  \DeclareMathOperator{\Ob}{Ob}
 \DeclareMathOperator{\Mor}{Mor}  
 \DeclareMathOperator{\diag}{diag} \DeclareMathOperator{\coker}{coker}

\numberwithin{equation}{section}

\newtheorem{theorem}{Theorem}
\numberwithin{theorem}{section}
\newtheorem{lemma}[theorem]{Lemma}
\newtheorem{proposition}[theorem]{Proposition}
\newtheorem{corollary}[theorem]{Corollary}

\theoremstyle{definition}

\newtheorem{example}[theorem]{Example}

\newcommand{\Z}{{\mathbb{Z}}}
\def\pt{*}

\def\tilde{\widetilde}

\def\smash{\wedge}

\def\isom{\cong}
\def\tensor{\otimes}

\def\into{\hookrightarrow}

\def\genby#1{{\langle #1\rangle}}

\newcommand{\Ast}{\mathop{\mathlarger{\mathlarger{\ast {} }} {} }}

\def\Loop{\Omega}

\def\wedgesum{\vee}

\def\op{\mathrm{op}}

\def\phi{\varphi}

\def\cat#1{\mathsf{#1}}

\def\Grpd{\cat{Grpd}}

\def\Mon{\cat{Mon}}

\def\Cat{\cat{Cat}}

\def\PtSet{\cat{Set}_*}

\def\sMon{\cat{sMon}}
\def\sSet{\cat{sSet}}

\def\FinOrd{\cat{\Delta}}
\def\Set{\cat{Set}}

\def\embed{\into}

\def\bar{\overline}

\def\Z{{\mathbb Z}}
\def\N{{\mathbb N}}

\def\barone{\bar{[1]}}
\def\barcy{\bar{[c_y]}}
\def\embed{\hookrightarrow}

\def\tensor{\otimes}

\def\co{\colon\thinspace}

\newarrow{Embed}{C}{-}{-}{-}{>}
\newarrow{Equals}{=}{=}{=}{=}{=}
\newarrow{Unique}{.}{.}{.}{.}{>}

\begin{document}

\title{Simplicial Monoid Actions and The Associated Universal Simplicial Monoid Construction\footnote{The authors gratefully acknowledge the assistance of Singapore Ministry of Education research grants AcRF Tier 1(WBS No. R-146-000-137-112) and AcRF Tier 2 (WBS No. R-146-000-143-112). The 2nd author is supported in part by a grant (No. 11028104) of NSFC of China.}}
\author{Man Gao and Jie Wu}
\date{}
\maketitle

\begin{abstract}
The reduced universal monoid on the action category associated to a pointed simplicial M-set has
appeared in the guise of various simplicial monoid and group constructions. These include the
classical constructions of Milnor and James, as well as their later generalizations by Carlsson and
Wu. We prove that, if any two n-simplices in the same orbit differ by the action of an invertible
monoid element, then the classifying space of this reduced universal monoid is the homotopy cofiber
of the inclusion from the pointed simplicial set into its reduced Borel construction. The known
formulae for the respective classifying spaces of the above four constructions are special cases of
this result. Thus, we unify categorially the four above-mentioned constructions.
\end{abstract}

%---------------section 1------------------------------------------------
\section{Introduction}

James' construction \cite{James55} on a pointed simplicial set $X$ is the simplicial monoid
obtained by applying the reduced free monoid functor in each dimension. Similarly, Milnor's
construction \cite{Milnor72} is the simplicial group obtained by applying the reduced free group
functor in each dimension. The geometric realization of both constructions is the loop-suspension
space $\Omega\Sigma |X|$ (see Section \ref{sec: smon actions}.)

The monoid or group structure facilitates the application of algebraic techniques to study the
homotopy and homology of the loop-suspension space. For example, the word length filtration on the
free monoid is used to prove James Splitting Theorem (see the original paper \cite{James55} or
pages 143-144 of \cite{Neisendorfer10}):
\[
\Sigma\Omega\Sigma |X| \simeq \bigvee_{n=1}^\infty |X|^{\smash n}
\]
Here $|X|^{\smash n}$ is the $n$-fold smash product of the pointed CW complex $|X|$.

Given the action of a discrete group on a pointed simplicial set $X$, Carlsson \cite{Carlsson84}
constructed a simplicial group whose classifying space is the homotopy cofiber of the inclusion of
$|X|$ into its reduced Borel construction (see Section \ref{sec: smon actions}.) Milnor's
construction is the special case of Carlsson's construction on trivial actions of the integers $\Z$
under addition. Carlsson also had an unreduced version of his simplicial group construction.

Carlsson's construction and its classifying space are of interest to homotopy theory. As Carlsson
noted \cite{Carlsson84}, certain quotients of a Thom complex by its lowest cell are the geometric
realization of his construction. This includes all stunted projective spaces. Wu \cite{Wu97}
constructed a simplicial monoid that depends on a data of a pointed simplicial set and a simplicial
monoid. If the simplicial monoid in the given data is the discrete monoid $\N$ of natural numbers
under addition, then Wu's construction reduces to James' construction. If the simplicial monoid in
the given data is a discrete group, then Wu's construction reduces to Carlsson's construction on
discrete group actions.

We observe that each of the above four constructions is in fact the reduced universal simplicial
monoid on the pointed simplicial action category associated to a pointed simplicial $M$-set. Let
$U\co\Cat\to \Mon$ be left adjoint to the fully faithful functor $\Mon\into \Cat$ that views a
monoid as a small category with one object. A {\textit{pointed category}} is a category with a
distinguished object as basepoint (see \ref{para: dunno}.) Given a small pointed category $C$,
write $\End_C(\pt)$ for the full subcategory of $C$ whose only object is the basepoint, the
{\textit{reduced universal monoid of $C$}} is defined as:
\begin{equation} \label{eq: red univ monoid def}
U[C] = \coker(U(\End_C(\pt)) \to U(C))
\end{equation}

View a right $M$-set as a functor $X\co M^\op\to \Set$. The {\textit{action category}} of $M$ is
the category of elements $X//M :=\int^{M^\op} \! X$. Explicitly, an object of $X//M$ is an element
of $X$ and a morphism $x\to y$ in $X//M$ is a monoid element $m\in M$ such that $y=xm$ (see Section
\ref{sec: smon actions}.) Define the {\textit{$J^M[X]$-construction on $X$}} as the monoid
$J^M[X]:=U[X//M]$. Write $J[\bullet]$ for the reduced free monoid functor (see \ref{para: red free
mon}). The $J^M[X]$-construction can be given explicitly as
$$\frac{J[X\smash M]}{\genby{\forall x\in X \forall m,k\in M\, (x\smash m)\cdot (xm\smash k)\sim (x\smash mk)}}.$$
In the simplicial case where $X$ is a pointed simplicial right $M$-set and $M$ is a simplicial
monoid, define $J^M[X]$ dimensionwise by setting $(J^M[X])_n=J^{M_n}[X_n]$.

Under a certain sufficient condition, we are able to compute the classifying space of $J^M[X]$. The
following theorem generalizes the results of Carlsson \cite{Carlsson84} and Wu \cite{James55}. As a
further consequence, this theorem implies that both the classifying spaces of Milnor's construction
and James' construction is $\Sigma |X|$.

\begin{theorem} \label{thm: cofiber seq for actions}
Let $X$ be a pointed right simplicial $M$-set where $M$ is a simplicial monoid. If for all $n$, for
all $x\in X_n$ and for all $m\in M_n$, there exists an invertible $k\in M_n$ such that $xm=xk$,
then $BJ^M[X]$ is the homotopy cofiber of the inclusion $|X|\into |X\rtimes_M EM|$.
\end{theorem}

To prove this theorem, we consider the setting of the reduced simplicial monoid on a general
pointed simplicial category. (Throughout this paper, a {\textit{simplicial category}} refers to a
simplicial object in $\Cat$.) We prove the following more abstract proposition. Recall that the
nerve of a simplicial category $C$ is a bisimplicial set whose geometric realization is the
{\textit{classifying space}} $BC$ (see \ref{para: nerve}.) Also recall that a category is
{\emph{totally disconnected}} if each connected component has only one object.
\begin{proposition} \label{prop: red geom realization of nerve of U(C)}
Let $C$ be a pointed simplicial category. If $C_n$ is equivalent to a totally disconnected category
for all $n$, then $BU[C]$ is the homotopy cofiber of $B\!\Ob(C)\to BC/B\End_C(\pt)$.
\end{proposition}

Theorem \ref{thm: cofiber seq for actions} is then a consequence of Proposition \ref{prop: red geom
realization of nerve of U(C)} in the case where $C=X//M$ is the pointed simplicial action category
associated to the simplicial monoid action of $M$ on $X$.

Finally, we remark that there is an unreduced version of all the above. In fact, the unreduced case
is used to prove the reduced case. In this introduction, we stated the results for the reduced
$J^M[X]$-construction. It is easier to work with the reduced construction in concrete applications
to homotopy and reduced homology. Some unreduced results can be found below.

%-------------------section 2-----------------------------------------------------------------------
\section{Preliminaries}

\paragraph{The Reduced Free Monoid and Group.} \label{para: red free mon} Let $J(S)$ denote the free monoid on a set $S$. If $S$ is pointed, its {\textit{reduced free monoid}} is $J[X]:=\coker(J(\pt)\to J(X))$. The reduced free monoid functor $J[\bullet]:\PtSet\to \Mon$ be the left adjoint of the forgetful functor $\Mon\embed\PtSet$ that views a monoid as a set whose basepoint is the identity element. Throughout, we use round brackets $(\bullet)$ to indicate an unreduced construction and square brackets $[\bullet]$ for a reduced construction. The letter $J$ is chosen in honor of James whose construction on a pointed simplicial set $X$ is obtained by the dimensionwise application of the reduced free monoid functor (see Section \ref{sec: smon actions}.)

Similarly, let $F(S)$ denote the free group on a set $S$. If $S$ is pointed, its {\textit{reduced free group}} is $F[X]:=\coker(F(\pt)\to F(X))$. The reduced free monoid functor $J[\bullet]\co\PtSet\to \Mon$ be the left adjoint of the forgetful functor $\Mon\embed\PtSet$ that views a monoid as a set whose basepoint is the identity element. The letter $F$ is chosen as Milnor's construction on a pointed simplicial set $K$ is commonly known as the $FK$-construction (see Section \ref{sec: smon actions}.)

\paragraph{Small Categories.} \label{para: small cat} A category $C$ is {\textit{small}} if the collection of morphisms $\Mor(C)$ is a set. Hence the collection of objects $\Ob(C)$ is also a set. It is an easy exercise to show that $\Ob$ commutes with coproducts of small categories. Write $t(f)$ and $s(f)$ for the source and target of a morphism $f\in \Mor(C)$ respectively. We use the convention that large categories are denoted in {\textsf{sans serif}}. A category in ordinary font is understood to be small.

\paragraph{Simplicial Sets.} A simplicial set is a presheaf on $\FinOrd$, whose objects are the finite linearly ordered sets $[n]=\{0< 1<\cdots< n\}$ and whose morphisms are weakly order-preserving maps. Recall that a poset $(P,\le)$ can be viewed as a poset category whose objects are the members of $P$ and there is a morphism $p\to q$ whenever $p\le q$. Viewing the linearly ordered set $[n]$ as a poset category thus gives a fully faithful functor $\FinOrd\into \Cat$.

\paragraph{Bisimplicial sets.} \label{para: sset} The geometric realization of a bisimplicial set is defined to be the geometric realization of its diagonal.  A map $X\to Y$ of bisimplicial sets is a {\textit{pointwise weak homotopy equivalence}} if for every $n$, the restriction $X_{n\bullet}\to Y_{n\bullet}$ is a weak homotopy equivalence of simplicial sets. If a map of bisimplicial sets $f\co X\to Y$ is a pointwise weak homotopy equivalence, then its diagonal $\diag f\co \diag X\to \diag Y$ is a weak homotopy equivalence of simplicial sets (see Chapter XII, Section 4 in \cite{BousfieldKan87}) and hence its geometric realization $|X|\to |Y|$ is a homotopy equivalence of CW complexes by Whitehead's Theorem.

\paragraph{Nerve.} \label{para: nerve} The {\textit{nerve}} of a small category $C$ is the simplicial set $NC$ whose $n$-simplices are the sequences
$x_0\to x_1\to\cdots\to x_n$ of composable morphisms in $C$. Dimensionwise application of the nerve functor to a simplicial category $C$ (which in this paper refers to a simplicial object in $\Cat$) yields a bisimplicial set $NC$ whose geometric realization is the {\textit{classifying space $BC$.}}

The nerve functor $N:\Cat\to\sSet$ is fully faithful. It is also a right adjoint and hence commute with all limits. The nerve also commute with arbitrary coproducts. The nerve functor also interacts well with the 2-categorial structure: if $\phi\dashv \psi$ is an adjoint pair of functors, then the simplicial map $N\phi$ is homotopy inverse to $N\psi$. In particular, an equivalence of categories $C\simeq D$ implies a homotopy equivalence $NC\simeq ND$.

\paragraph{The Classifying Space and The Wedge Sum.} \label{para: dunno} A pointed object in a category $\cat{C}$ is a morphism in $\cat{C}$ whose source is the terminal object. Let $\cat{C}_*$ denote the category of pointed objects in $\cat{C}$. The coproduct in $\cat{C}_*$ is called the {\textit{wedge sum}}, written $\bigvee$. Monoids are uniquely pointed, hence the wedge sum of pointed monoids coincides the coproduct of monoids: this is usually called the {\textit{free product}} $\Ast$. The classifying space of topological monoids commutes, up to homotopy, with the wedge sum (see \cite{Fiedorowicz84}). We will use this result in the following form: given a family of monoids $\{M_i\}$, each of which is viewed as a category with only one object, the inclusions $M_j\to\Ast M_i$ induce a weak homotopy equivalence of simplicial sets:
\[
\bigvee NM_i \to N(\Ast M_i)
\]

\paragraph{Looping the Classifying Space} \label{para: loop nerve sgrp} For every simplicial group $H$, there is a homotopy equivalence
\[
|H|\simeq \Loop BH
\]
This can be proven as follows. There is a weak homotopy equivalence $G(\bar{W}H) \to H$ where $G$ is Kan's construction and $\bar{W}$ is the bar construction (see Page 270 in \cite{GoerssJardine99}). There is another weak homotopy equivalence $\diag NH \to \bar{W} H$ (this is Proposition 3.1 of \cite{Wu97}). Combining these gives a weak homotopy equivalence $G(\diag NH)\to H$. The geometric realization of Kan's construction is the loop functor, hence taking geometric realization gives the result.

%------------------section 3--------------------------------------------------------------
\section{The Universal Monoid and Its Classifying Space}

The unit map $\eta_C:C\to U(C)$ is the cokernel of $\Ob(C)\to C$, viewing $\Ob(C)$ as a discrete subcategory of $C$. Hence $N\eta_C:NC\to NU(C)$ factors through the quotient simplicial set $NC/N\!\Ob(C)$. This induces a map
\begin{equation} \label{eq: THE old map}
NC/N\!\Ob(C)\to NU(C).
\end{equation}
For any set $S$, there is an isomorphism $NS\isom S$ where $S$ is viewed as a discrete category when taking the nerve and is viewed as a discrete simplicial set on the right-hand side. This is proven by observing that a set is a coproduct of singletons and that the nerve functor commutes with coproducts (see \eqref{para: nerve}), so this reduces to the obvious fact that the standard 0-simplex $N*=N[0]$ is a singleton in each dimension.

In particular $N\!\Ob(C)\isom \Ob(C)$. Via this isomorphism, the map \eqref{eq: THE old map} becomes the following map which we shall call $N\eta_C$ by abuse of notation
\begin{equation} \label{eq: THE final map}
N\eta_C:NC/\Ob(C)\to NU(C).
\end{equation}

\begin{lemma} \label{lem: nerve unit map of monoid}
If the category $M$ has only one object, then $N\eta_M:NM/\Ob(M)\to NU(M)$ is an isomorphism of simplicial sets.
\end{lemma}

\begin{proof}
Since $M$ has only one object, the discrete simplicial set $\Ob(M)$ is isomorphic to the standard 0-simplex. The proof is complete by noting that the unit map $\eta_M:M\to U(M)$ is an isomorphism since $M$ has only one object.
\end{proof}

Let $\bar{\bullet}:\Cat\to \Grpd$ be the left adjoint of the fully faithful functor $\Grpd\into \Cat$ that views a small groupoid as a small category. The groupoid $\bar{C}$ is said to be the {\textit{groupoid completion}} of a small category $C$.

\begin{lemma} \label{lem: nerve unit map of barone}
The simplicial map $N\eta_{\bar{[1]}}:N\barone/ \Ob(\barone)\to NU(\bar{[1]})$ is a weak homotopy equivalence.
\end{lemma}

\begin{proof}
The category $[1]$ has two objects $0$ and $1$ and a unique nonidentity morphism $t:0\to 1$. Hence $\Ob(\barone)=\{0,1\}$ and $U(\bar{[1]})$ is the free group generated by $t$. Consider the following commutative diagram of simplicial sets:
\begin{diagram}
\frac{N\barone}{\{0,1\}} & \rTo^{N\eta_{\barone}} & NU(\barone) \\
\uEmbed<{j}                  &                        & \uTo>{g}      \\
\frac{N[1]}{\{0,1\}}    & \rEquals               & S^1.
\end{diagram}
Here $j$ is induced by $[1]\embed \barone$ while $g$ sends the unique nondegenerate 1-simplex of the simplicial circle $S^1$ to the 1-simplex $t$ of $\barone$.

To show that $N\eta_{\barone}$ is a weak homotopy equivalence, it suffices to show that $j$ and $g$ are both weak homotopy equivalences. First consider $j$. One one hand the standard 1-simplex $N[1]$ is contractible; on the other hand the equivalence of categories $\barone\simeq [0]$ implies the homotopy equivalence $N\barone\simeq N[0]$. This shows that both $N[1]$ and $N\barone$ are weakly contractible, hence $j$ is a weak homotopy equivalence.

Next consider $g$. The simplicial circle $S^1$ is a $K(\Z,1)$. Since $U(\barone)\isom \Z$ and $g$ sends the unique nondegenerate 1-simplex of $S^1$ to a generator of $U(\barone)$, then $g$ is a homotopy equivalence. This completes the proof.
\end{proof}

Let $\phi:M\to C$ be an equivalence of categories where $M$ has only one object $\bullet$. The essentially surjectivity of $\phi$ implies that any two objects of $C$ are isomorphic to $\phi(\bullet)$ and hence to each other. In particular, given a fixed object $x$ in $C$, there exists isomorphisms $c_y:x\to y$ for every object $y$ in $C$.

\begin{proposition} \label{prop: wedgesum decomposition of cat}
Fix a basepoint $x$ of a category $C$.

If $C$ is equivalent to a category with one object, then for each family $\{c_y:x\to y|\, y\in \Ob(C), y\neq
x\}$ of isomorphisms in $C$, then the inclusions $\End_C(x)\into C$ and $\bar{[c_y]}\into C$ induce an isomorphism of pointed categories:
\[
\End_C(x)\vee\bigvee_{y\neq x}\bar{[c_y]}\isom C.
\]
Here $\bar{[c_y]}\isom \barone$ is the category whose objects are $x$ and $y$ and whose nonidentity morphisms are $c_y$ and $c_y^{-1}$.
\end{proposition}

\begin{proof}
Set $c_x$ as the identity morphism of $x$. We verify the universal property. Let $D$ be a category and functors $\alpha:\End_C(x)\to D$ and $\beta_y:\bar{[c_y]}\to D$ be given.
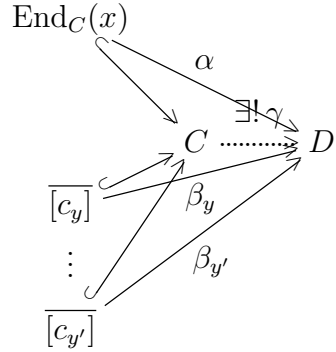
\begin{diagram}
\End_C(x)       &  \\
                &  \rdEmbed \rdTo(4,2)^{\alpha} & \\
                &                               & C & \rUnique^{\exists!\, \gamma}              & D  \\
   \bar{[c_y]}  & \ruEmbed(2,1) \ruEmbed(2,3)   &   & \ruTo(4,1)_{\beta_y} \ruTo(4,3)_{\beta_{y'}} & \\
   \vdots       &                               &   &                       & \\
 \bar{[c_{y'}]} &
\end{diagram}

Construct the functor $\gamma$ as follows. For any object $y\neq x$, define $\gamma(x)=\beta_y(y)$ (being a based functor $\gamma$ must send $x$ to the basepoint of $D$.) For any morphism $y\xrightarrow{d}z$ in $C$, define
\[
\gamma(d)=\beta_z(c_z)\circ \alpha(c_z^{-1}\circ d\circ c_y)\circ\beta_y(c_y^{-1}).
\]
It is routine to check that $\gamma$ is makes the required diagrams commute and is unique.
\end{proof}

\begin{lemma} \label{lem: thm in case of connected cat}
If a category $C$ is equivalent to a category with one object, then $N\eta_C:NC/\Ob(C)\to NU(C)$ is a weak homotopy equivalence.
\end{lemma}

\begin{proof}
Fix a basepoint $x\in \Ob(C)$. By the remarks just before Proposition \ref{prop: wedgesum decomposition of cat}, the assumption that $C$ is equivalent to a category with one object implies the existence of an isomorphisms $c_y:x\to y$ for each object $y\neq x$.
Consider the commutative diagram of simplicial sets:
\begin{diagram}
\frac{NC}{\Ob(C)}                                                     & \rTo^{N\eta_{C}}                    & NU(C)                 \\
\uTo<{i}                                                            &                                     & \uTo>{NU(\iota)}             \\
                                                                       &                               & NU(\End_C(x)\vee\bigvee_{y\neq x}\barcy)\\
\frac{N\End_C(x)\wedgesum \bigvee_{y\neq x} N\barcy}{\{x\}\wedgesum \bigvee_{y\neq x}\{x,y\}} &          & \uTo>{N\psi}                 \\
\uTo<{f}                                                              &                      & N\left[U(\End_C(x))\ast\Ast_{y\neq x} U(\barcy)\right] \\
                                                                      &                                      & \uTo>{h}             \\
\frac{N\End_C(x)}{\{x\}} \wedgesum \bigvee_{y\neq x} \frac{N\barcy}{\{x,y\}} & \rTo^{N\eta_{\End_C(x)} \vee\bigvee N\eta_{\barcy}}
 & NU(\End_C(x))\vee\bigvee_{y\neq x} NU(\barcy)
\end{diagram}
To show that $N\eta_{C}$ is a weak homotopy equivalence, it suffices to show that every other map is a weak homotopy equivalence. We will describe each map in turn and prove that it is a weak homotopy equivalence.

The isomorphisms $\Ob(C)\isom \{x\}\vee\bigvee_{y \neq x}\{x,y\}$ and the inclusions $\End_C(x)\into C$ and $\barcy\into C$ induce the map $i$. As noted above, the assumption that $C$ is equivalent to a category with one object implies that any two objects of $C$ are isomorphic. Hence the fully faithful functor $\End_C(x)\into C$ is essentially surjective. In other words, there is an equivalence of categories $\End_C(x)\simeq C$ which gives a homotopy equivalence $N\End_C(x)\simeq NC$. Also $\barcy\isom [1]\simeq [0]$, which gives another homotopy equivalence $N\barcy\simeq N[0]=\pt$. Therefore $i$ is a weak homotopy equivalence.

The map $f$ is an isomorphism because the cokernel commutes with the wedge sum, both of them being colimits.

The map $N\eta_{\End_C(x)} \vee\bigvee_{y\neq x} N\eta_{\barcy}$ is a weak homotopy equivalence. This is because $N\eta_{\End_C(x)}$ is an isomorphism by Lemma \ref{lem: nerve unit map of monoid} and $N\eta_{\barcy}$ is a weak homotopy equivalence by Lemma \ref{lem: nerve unit map of barone}.

Recall that the nerve of monoids commutes with wedge sums, up to weak homotopy equivalence (see \ref{para: dunno}.) Since $U(\End_C(x))$ and $U(\barcy)$ are monoids, the map $h$ is a weak homotopy equivalence.

The functor $U$ is a left adjoint and thus commutes with colimits. Thus $\psi:U(\End_C(x))\ast\Ast_{y\neq x} U(\barcy)\to U(\End_C(x)\vee\bigvee_{y\neq x}\barcy)$ is an isomorphism and so is $N\psi$.

By Proposition \ref{prop: wedgesum decomposition of cat}, the functor $\iota:\End_C(x)\vee\bigvee_{y\neq x}\barcy\to C$ is an isomorphism. Hence $NU(\iota)$ is a isomorphism.
\end{proof}

\begin{proposition} \label{prop: unpointed coprod of cofiber}
Let $\{A_i\subset X_i\}$ be a collection where $A_i$ is a simplicial subset of $X_i$.  There is a isomorphism of simplicial sets:
$$\frac{\coprod X_i}{\coprod A_i}\isom \bigvee \frac{X_i}{A_i}.$$
\end{proposition}

\begin{proof}
Define the pointed simplicial set $X_i^+$ by adjoining a disjoint basepoint to $X_i$. Let $A_i^+$ be the pointed simplicial subset of $X_i^+$ that contains $A_i$ and this disjoint basepoint. The cokernel and the wedge sum commutes, both being colimits:
$$\frac{\bigvee X_i^+}{\bigvee A_i^+}\isom \bigvee \frac{X_i^+}{A_i^+}.$$
But $\frac{\coprod X_i}{\coprod A_i}$ is isomorphic to $\frac{\bigvee X_i^+}{\bigvee A_i^+}$ and $\bigvee \frac{X_i}{A_i}$ is isomorphic to $\bigvee \frac{X_i^+}{A_i^+}$, so this completes the proof.
\end{proof}

\begin{lemma} \label{lem: thm in case of one level cat}
If the category $C$ is equivalent to a totally disconnected category,
then $N\eta_C:NC/\Ob(C)\to NU(C)$ is a weak homotopy equivalence.
\end{lemma}

\begin{proof}
Decompose $C=\coprod_i A_i$ into its connected components. The assumption that $C$ is equivalent to a totally disconnected category implies that each connected component $A_i$ is equivalent to a category with one object. Thus Lemma \ref{lem: thm in case of connected cat} gives weak homotopy equivalences $N\eta_{A_i}:NA_i/\Ob(A_i)\to NU(A_i)$ for all $i$. Consider the following commutative diagram:
\begin{diagram}
\frac{N(\coprod A_i)}{\Ob(\coprod A_i)} & \rTo^{N\eta_C}               & NU(\coprod A_i) \\
\uTo<{f}                                 &                              & \uTo>{k}            \\
\frac{\coprod NA_i}{\coprod \Ob(A_i)}   &                              & N(\Ast U(A_i))     \\
\uTo<{g}                                 &                              & \uTo>{h}             \\
\bigvee \frac{NA_i}{\Ob(A_i)}           & \rTo^{\bigvee N\eta_{A_i}}   & \bigvee NU(A_i).
\end{diagram}

To show the required weak homotopy equivalence $N\eta_C$, it suffices to show that every other map is a weak homotopy equivalence. Since $\Ob$ and $N$ commute with coproducts (see \ref{para: small cat} and \ref{para: nerve} respectively,) the map $f$ is an isomorphism. The map $g$ is an isomorphism by Proposition \ref{prop: unpointed coprod of cofiber}. The map $h$ is a weak homotopy equivalence (see \ref{para: dunno}.) The map $k$ is an isomorphism since $U$ is a left adjoint and hence commutes with coproducts.
\end{proof}

\begin{proposition}
Let $C$ be a simplicial category. If $C_n$ is equivalent to a totally disconnected category for all $n$, then $BU(C)$ is the homotopy cofiber of $B\! \Ob(C)\to BC$.
\end{proposition}

\begin{proof}
By the remarks before Lemma \ref{lem: nerve unit map of monoid}, for any set $S$, there is a isomorphism $NS\isom S$. For a simplicial set $X$, this gives a pointwise isomorphism $NX\isom X$ of bisimplicial sets, hence an isomorphism $BX=|NX|\isom |X|$ of CW complexes.

Now, the bisimplicial set $NC/\Ob(C)$ is the homotopy cofiber of $\Ob(C)\to NC$, hence $|NC/\Ob(C)|$ is the homotopy cofiber of $|\Ob(C)|\to |NC|=BC$. The previous paragraph gives an isomorphism $B\!\Ob(C)\isom |\Ob(C)|$, therefore $|NC/\Ob(C)|$ is also the homotopy cofiber of $B\!\Ob(C)\to BC$.

By Lemma \ref{lem: thm in case of one level cat}, the assumption that each $C_n$ is equivalent to a totally disconnected category implies that the map of bisimplicial sets $N\eta_C: NC/\Ob(C)\to NU(C)$ is a pointwise weak homotopy equivalence. Taking geometric realization, we obtain a homotopy equivalence $|N\eta_C|: |NC/\Ob(C)|\simeq |NU(C)|=BU(C)$ of CW complexes (see \ref{para: sset}.) Therefore $BU(C)$ is the homotopy cofiber of $B\!\Ob(C)\to BC$.
\end{proof}

%----------------------------------section 4----------------------------------------
\section{The Reduced Version}

We now move on to give a reduced version of the above. Recall from \eqref{eq: red univ monoid def}, that the reduced universal monoid $U[C]$ on a pointed category $C$ is defined to be $\coker(U(\End_C(\pt) ) \to U(C))$ . The reduced monoid $U[C]$ can also be presented by generators and relations:
\[
\frac{J \left[\Mor(C)/\End_C(\pt)\right]}{\genby{\forall X\in \Ob(C)\, (1_X\sim 1), \forall f,g\in \Mor(C) \, (t(f)=s(g)
\implies g\circ f\sim f\cdot g \ )}}.
\]

\begin{proposition} \label{prop: univ prop of red universal mon}
Let $C$ be a pointed category. The reduced universal monoid $U[C]$ has the universal property:
\begin{diagram}
C & \rTo & U[C] \\
 & \rdTo<{\psi} & \dUnique>{\exists!\tilde{\psi}} \\
 &              & M
\end{diagram}
Given a functor $\psi: C\to M$ to a monoid $M$, viewed as a category with one object, such that $\psi$ sends every endomorphism of $\pt$ to
$1_M$, there exists a unique monoid homomorphism
$\tilde{\psi}:U[C]\to M$ such that the above diagram commutes. Here $C\to U[C]$ is the composite $C\xrightarrow{\eta_C} U(C)\to U[C]$.
\end{proposition}

By the remarks before Lemma \ref{lem: nerve unit map of monoid}, for a category $C$, the unit map $\eta_C:C\to U(C)$ induces a simplicial map $N\eta_C:NC/\Ob(C)\to U(C)$. If $C$ is pointed, the composite $C\xrightarrow{\eta_C}U(C)\to U[C]$ sends every endomorphism of the basepoint to the identity, hence $N\eta_C:NC/\Ob(C)\to U(C)$ factors through $\frac{NC/N\End_C(\pt)}{\Ob(C)}$. By further abuse of notation, we will still call the induced map $N\eta_C:\frac{NC/N\End_C(\pt)}{\Ob(C)}\to NU[C]$.

\begin{lemma} \label{lem: red thm in case of connected cat}
Let $C$ be a pointed category. If $C$ is equivalent to a category with just one object, then $N\eta_C:\frac{NC/N\End_C(\pt)}{\Ob(C)}\to NU[C]$ is a weak homotopy equivalence.
\end{lemma}

\begin{proof}
Call the basepoint of $C$ as $x$. Since $C$ is equivalent to a one-object category, we may choose isomorphisms $c_y:x\to y$ in $C$ for each object $y\neq x$.
Consider the commutative diagram:
\begin{diagram}
\frac{NC/N\End_C(x)}{\Ob(C)}                                        & \rTo^{N\eta_{C}}                    & NU[C]                 \\
\uTo<{i}                                                            &                                     & \uTo>{NU(\phi)}             \\
                                                                       &                                     & \\
\frac{\bigvee_{y\neq x} N\barcy}{\bigvee_{y\neq x}\{x,y\}}          &                                     & NU(\Ast_{y\neq x} \barcy)                \\
\uTo<{f}                                                              &                                      & \\
                                                                      &                                      & \uTo>{h}             \\
\bigvee_{y\neq x} \frac{N\barcy}{\{x,y\}}                          &  \rTo^{\bigvee_{y\neq x} N\eta_{\barcy}} & \bigvee_{y\neq x} NU(\barcy).
\end{diagram}
To show that $N\eta_{C}$ is a weak homotopy equivalence, it suffices to show that every other map is a weak homotopy equivalence. We will describe each map in turn and prove that it is a weak homotopy equivalence.

The isomorphism $\Ob(C)\isom \bigvee_{y \neq x}\{x,y\}$ and the inclusions $\barcy\into C$ induce the map $i$. As before, the assumption that $C$ is equivalent to a category with just one object implies the equivalence of categories $\End_C(x)\simeq C$. Thus $N\End_C(x)\simeq NC$ so that $NC/N\End_C(x)\simeq \pt$. Also the equivalence of categories $\barcy\isom\barone\simeq [0]$ implies $N\barcy\simeq  N[0]=\pt$. Therefore $i$ is a weak homotopy equivalence.

The map $f$ is the isomorphism of simplicial sets since the cokernel and the wedge sum commute, both being colimits.

The map $\bigvee_{y\neq x} N\eta_{\barcy}$ is a weak homotopy equivalence since each $N\eta_{\barcy}$ is a weak homotopy equivalence by Lemma \ref{lem: nerve unit map of barone}.

The map $h$ is is the following composite
$$\bigvee_{y\neq x} NU(\barcy)\to N(\Ast_{y\neq x} U(\barcy))\to NU(\Ast_{y\neq x} \barcy).$$
The first map is a weak homotopy equivalence since the nerve of monoids commutes with wedge sum up to weak equivalence (see \eqref{para: dunno}.) The second map is an isomorphism since $U$, being a left adjoint, commutes with colimits. Therefore $h$ is a weak homotopy equivalence.

The functor $\phi$ is induced by the inclusions $\barcy\into C$. These inclusions and $\End_C(x)\into C$ induces an isomorphism $\End_C(x)\vee\bigvee_{y\neq x}\bar{[c_y]}\isom C$ by Proposition \ref{prop: wedgesum decomposition of cat}. Since $U$ commutes with colimits, this gives $U(\End_C(x))\ast \Ast_{y\neq x} U(\barcy)\isom U(C)$. But $U[C]=\coker(U(\End_C(x))\to U(C))$, hence $\phi:\Ast_{y\neq x} U(\barcy)\to U[C]$ is an isomorphism. Therefore $N\phi$ is an isomorphism of simplicial sets.
\end{proof}

\begin{proposition} \label{prop: how red U commute with coprod}
Let $C$ be a pointed category and let $C=A_0\sqcup\coprod A_i$ be a decomposition of $C$ into its connected components where $A_0$ contains the basepoint. Then there is a monoid isomorphism:
$$U[C]\isom U[A_0]\ast\Ast U(A_i).$$
\end{proposition}

\begin{proof}
Note that $\End_C(\pt)=\End_{A_0}(\pt)$ as $A_0$ contains the basepoint. Hence
\begin{align}
&\coker(U(\End_{C}(\pt))\to [U(A_0)\ast\Ast U(A_i)]) \label{eq: what should be U(C)}\\
\isom {} &\coker(U(\End_{A_0}(\pt))\to U(A_0))\ast\Ast U(A_i). \notag\\
= {} &U[A_0]\ast\Ast U(A_i) \notag
\end{align}
But $U(C)=U(A_0\sqcup\coprod A_i)\isom U(A_0)\ast\Ast U(A_i)$ since $U$ is a left adjoint and hence commutes with coproducts. Thus \eqref{eq: what should be U(C)} is just $\coker(U(\End_C(\pt) ) \to U(C))=U[C]$, which completes the proof.
\end{proof}

\begin{lemma} \label{lem: red thm in case of one level cat}
Let $C$ be a pointed category. If $C$ is equivalent to a totally disconnected category,
then $N\eta_C:\frac{NC/N\End_C(\pt)}{\Ob(C)}\to NU[C]$ is a weak homotopy equivalence.
\end{lemma}

\begin{proof}
Decompose $C=A_0\sqcup \coprod_{i\neq 0} A_i$ into its connected components where $A_0$ contains the basepoint $\pt$. Since $C$ is equivalent to a totally disconnected category by assumption, each connected component $A_i$ is equivalent to a category with just one object. Thus Lemma \ref{lem: red thm in case of connected cat} gives a weak homotopy equivalence $N\eta_{A_0}:\frac{NA_0/N\End_{A_0}(\pt)}{\Ob(A_0)}\to NU[A_0]$ in the reduced case while Lemma \ref{lem: thm in case of connected cat} gives weak homotopy equivalences $N\eta_{A_i}:\frac{NA_i}{\Ob(A_i)}\to NU(A_i)$ for $i\neq 0$ in the unreduced case. Consider the following commutative diagram:
\begin{diagram}
\frac{N(A_0\sqcup\coprod_{i\neq 0} A_i)/N\End_C(\pt)}{\Ob(A_0\sqcup\coprod_{i\neq 0} A_i)} & \rTo^{N\eta_C} & NU[A_0\sqcup\coprod_{i\neq 0} A_i] \\
\uTo<{f}                                 &                              & \uTo>{k}            \\
\frac{(NA_0\sqcup\coprod_{i\neq 0} NA_i)/N\End_{A_0}(\pt)}{\Ob(A_0)\sqcup\coprod_{i\neq 0} N\!\Ob(A_i)}   &                              & N(U[A_0]\ast\Ast_{i\neq 0} U(A_i))     \\
\uTo<{g}                                 &                              & \uTo>{h}             \\
\frac{NA_0/N\End_{A_0}(\pt)}{\Ob(A_0)}\vee\bigvee_{i\neq 0} \frac{NA_i}{\Ob(A_i)} & \rTo^{N\eta_{A_0}\vee\bigvee N\eta_{A_i}} & NU[A_0]\vee\bigvee_{i\neq 0} NU(A_i)
\end{diagram}

To show that $N\eta_C$ is a weak homotopy equivalence, it suffices to show that every other map is a weak homotopy equivalence. The map $f$ is an isomorphism since $N\End_C(\pt)=N\End_{A_0}(\pt)$ and both $\Ob$ and $N$ commute with coproducts (see \ref{para: small cat} and \ref{para: nerve} and respectively). The map $g$ is an isomorphism by Proposition \ref{prop: unpointed coprod of cofiber}. The map $h$ is a weak homotopy equivalence (see \ref{para: dunno}.) The map $k$ is an isomorphism by Proposition \ref{prop: how red U commute with coprod}.
\end{proof}

The proof of Proposition \ref{prop: red geom realization of nerve of U(C)} then follows exactly as in the unreduced case.

%------------------------section 5-------------------------------------------------------------
\section{Proof of Theorem \ref{thm: cofiber seq for actions}} \label{sec: smon actions}

Let $X$ be a right $M$-set where $M$ is a monoid. Recall from the introduction that its
{\textit{action category}} $X//M$ is the category of elements $\int^{M^\op} \! X$ of the functor
$X\co M^\op\to \Set$. Explicitly, an object of $X//M$ is an element of $X$ and a morphism $x\to y$
in $X//M$ is a monoid element $m\in M$ such that $y=xm$. We denote this morphism by $m\co x\to xm$.
Given the action of a simplicial monoid $M$ on a simplicial set $X$, its {\emph{simplicial action
category of $X$}}, also written as $X//M$, is defined dimensionwise by setting
$(X//M)_n:=X_n//M_n$. If the action is pointed, that is to say, the set $X$ has a basepoint which
is fixed by the $M$-action, then the action category $X//M$ is a pointed category whose
distinguished object is the basepoint of $X$.

If $M$ is a discrete monoid that acts on a simplicial set $X$, then $N(X//M)$ is the homotopy colimit of the functor $X\co M^\op \to \sSet$. (More accurately, this is the uncorrected homotopy colimit \cite{Shulman09} or the weak 2-colimit {\cite{nLabAction, nLabLimit}})

\begin{lemma} \label{lem: criterion for iso in act cat}
Let $X$ be a $M$-set where $M$ is a monoid. A morphism $m\co x\to xm$ in $X//M$ is an isomorphism if and only if $m$ is invertible in $M$.
\end{lemma}

\begin{proof}
For the forward direction, suppose $k\co xm\to x$ is an inverse of the morphism $m\co x\to xm$. Then $mk\co x\to x$ and $km\co y\to y$ are identity morphisms so that $mk=km=1$, that is to say, the elements $m$ and $k$ are inverses in $M$.

For the backward direction, suppose that $k$ is the inverse of $m$ in $M$. Then $k\co xm\to xmk=x$ is the inverse of the morphism $m\co x\to xm$.
\end{proof}

\begin{proposition} \label{prop: crit act cat equiv to one obj}
Let $X$ be a right $M$-set where $M$ is a monoid. The following are equivalent:
\begin{enumerate}
\item The action category $X//M$ is equivalent to a totally disconnected category.
\item For all $x\in X$ and $m\in M$, there exists an invertible $k\in M$ such that $xm=xk$.
\end{enumerate}
\end{proposition}

\begin{proof}
$1\implies 2$. Let $x\in X$ and $m\in M$ be given. If $m$ fixes $x$, then take $k$ to be the identity of $M$. Otherwise $x$ and $xm$ are distinct objects of $X//M$ connected by the morphism $m\co x\to xm$. The assumption that $X//M$ is equivalent to a totally disconnected category implies that any two objects in the same connected component are isomorphic. By the previous Lemma \ref{lem: criterion for iso in act cat}, this implies that there exists $k$ is invertible in $M$ such that $xm=xk$.

$2\implies 1$. Suppose that for all $x\in X$ and $m\in M$, there exists an invertible $k\in M$ such that $xm=xk$. It suffices to show that if two distinct objects $z$ and $w$ are connected by a single morphism $m\co z\to w$, then $z$ and $w$ are isomorphic in $X//M$. This would imply that any two distinct objects $x$ and $y$ in the same connected component of $X//M$ are connected by a finite sequence of isomorphisms and hence are isomorphic, thus $X//M$ is equivalent to a totally disconnected category.

Indeed, let two distinct objects $z$ and $w$ connected by a single morphism $m\co z\to w$ be given. By assumption, there there exists an invertible $k\in M$ such that $w=zm=zk$. Lemma \ref{lem: criterion for iso in act cat} implies that $k\co z\to zk=w$ is an isomorphism. This completes the proof.
\end{proof}

Denote by $WM$ any contractible simplicial set with a free $M$-action. Any two such simplicial sets are equivariantly homotopy equivalent. Let $EM:=|WM|$ denote the geometric realization of $WM$. The {\textit{bar construction of $G$}} is the orbit space $\bar{W}M:=WM/M$. The {\textit{classifying space of $M$}} is the geometric realization $BM:=|\bar{W}M|$. In fact $B(M//M)\simeq BM$ where $M//M$ is the action category of the right translation action on $M$(see Lemma 4.2 of \cite{BousfieldKan87}.) If $G$ is a discrete group, its classifying space $BG$ is the Eilenberg-Mac Lane space $K(G,1)$. View a simplicial monoid $M$ as a simplicial category with one object in each dimension, there is a homotopy equivalence $|NM|\simeq BM$ (see Section 6.2 in \cite{GaoThesis}.) This warrants the definition that the classifying space of a simplicial category $C$ is $BC:=|NC|$ (see \ref{para: nerve}.)

Consider the action of $M$ on a simplicial set $X$. The free simplicial $M$-set associated to $X$ is $X\times WM$ with the diagonal action. The {\textit{Borel construction of $X$}} is the orbit space $X\times_G WM:=(X\times WM)/M$. For example, the Borel construction of the $M$-action on the standard 0-simplex $\Delta[0]=*$ is bar construction of $M$:
\[
\pt\times_M WM\simeq \bar{W}M
\]
Suppose the $M$-action is pointed, that is to say, the simplicial set $X$ has a basepoint and the $M$-action fixes the basepoint. The {\textit{reduced Borel construction}} of this pointed action, written $X\rtimes_M WM$, is the homotopy cofiber of $\pt\times_M WM\to X\times_M WM$.

\begin{proof}[Proof of Theorem \ref{thm: cofiber seq for actions}]
Consider the pointed simplicial action category $X//M$. By definition, we have $\Ob(X//M)=X$. The remarks just before Lemma \ref{lem: nerve unit map of monoid} implies
\begin{equation} \label{eq: expression for ob X//M}
B\!\Ob(X//M)=|N\!\Ob(X//M)|=|NX|=|X|.
\end{equation}

By definition, the reduced Borel construction $X\rtimes_M EM$ is the homotopy cofiber of $\pt\times_M EM\to X\times_M EM$. Thus, it is also the homotopy cofiber of $N(\pt//M)\to N(X//M)$. Note that $\pt//M=\End_{X//M}(\pt)$. Therefore $X\rtimes_M EM \simeq N(X//M)/N\!\End_{X//M}(\pt)$. Take the geometric realization to obtain the homotopy equivalence homotopy equivalence
\begin{equation} \label{eq: expression for X//M}
|X\rtimes_M EM| \simeq B(X//M)/B\End_{X//M}(\pt).
\end{equation}

Finally, a routine verification of the universal property Proposition \ref{prop: univ prop of red universal mon} gives the monoid isomorphism:
\begin{equation} \label{eq: U[X//M] and JMX}
J^M[X] \isom U[X//M]
\end{equation}

By assumption, for all $n$, for all $x\in X_n$ and $m\in M_n$, there exists an invertible $k\in M_n$ such that $xm=xk$. Hence Proposition \ref{prop: crit act cat equiv to one obj} implies that $(X//M)_n$ is equivalent to a totally disconnected category. By Proposition \ref{prop: red geom realization of nerve of U(C)}, the classifying space $BU[X//M]]$ is the homotopy cofiber of $B\!\Ob(X//M)\to B(X//M)/B\End_{X//M}(\pt)$. By \eqref{eq: expression for ob X//M} and \eqref{eq: expression for X//M} and \eqref{eq: U[X//M] and JMX}, we see that $BJ^M[X]$ is the homotopy cofiber of $|X|\to |X\rtimes_M EM|$. This completes the proof.
\end{proof}

%------------------section  6-----------------------------------------------------------

\section{Examples}

The following example gives a family of simplicial monoid actions each of which satisfies the sufficient condition in Theorem \ref{thm: cofiber seq for actions}. This family contains all trivial simplicial monoid actions and all simplicial group actions. This example shows that Theorem \ref{thm: cofiber seq for actions} strictly generalizes the results of Carlsson and Wu.

\begin{example} \label{eg: MxG example}
Begin with the action of a simplicial group $G$ on a pointed simplicial set $X$. Consider the product $M\times G$ of a simplicial monoid $M$ with $G$. Extend the action by letting $(m,1)$ act trivially for all $m\in M_n$. Since $x(m,g)=x(m,1)(1,g)=x(1,g)$ and $(1,g)$ is invertible, this $(M\times G)$-action satisfies the sufficient condition in Theorem \ref{thm: cofiber seq for actions}. Thus $BJ^{M\times  G}[X]$ is the homotopy cofiber of $|X|\to |X\rtimes_{M\times G} E(M\times G)|$.

Set $G=1$ to obtain a trivial simplicial monoid action; set $M=1$ to obtain a simplicial group action. However, in general, this $(M\times G)$-action is neither a simplicial group action nor a trivial simplicial monoid action.
\end{example}

Given a simplicial group $G$ that acts on a pointed simplicial set $X$, the pointed simplicial action category $X//G$ is in fact a groupoid. One checks that $U[X//G]$ is in fact a group, hence $J^M[X]$ can be alternately defined as the {\emph{reduced universal group}} on $X//G$. We leave the reader to formulate the precise definition (or see \cite{GaoWu}.)

Given a pointed simplicial $G$-set $X$ where $G$ is a {\emph{discrete}} simplicial group, Carlsson \cite{Carlsson84} defined a simplicial group, which is essentially the $J^G[X]$-construction. We generalize Carlsson's result in \cite{Carlsson84} to a general simplicial group.

\begin{corollary} \label{cor: Carlsson for sgrp}
Let $X$ be a pointed simplicial $G$-set where $G$ is a simplicial group. Then there is a homotopy equivalence:
$$|J^G[X]|\simeq\Omega(|X\rtimes_GEG|/|X|).$$
\end{corollary}

\begin{proof}
Since every element of $G_n$ is invertible, Theorem \ref{thm: cofiber seq for actions} applies to give $BF^M[X] \simeq |(X\rtimes_G EG)/X|$. Since $J^G[X]$ is a simplicial group, apply \eqref{para: loop nerve sgrp} to obtain $|J^G[X]|\simeq \Loop BJ^G[X]$. Combine these two homotopy equivalences to obtain the required result.
\end{proof}

Set $G=1$ in the above example gives a result of Wu \cite{Wu97}. The category $\sMon$ of simplicial monoids is simplicially enriched over the category $\sSet_*$ of pointed simplicial sets. The tensor product of a simplicial monoid $M$ and a pointed simplicial set $X$ is defined by the following natural isomorphism in simplicial monoids $K$:
\begin{equation} \label{eq: adjoint def of tensor prod}
{\sMon}(X\tensor M, K) \isom {\sSet_*}(X,\sMon(M,K))
\end{equation}
An exposition of the tensor product can be found in \cite{MutluPorter98}.

For each $x\in X_n$, let $(M_n)_x$ denote an isomorphic copy of $M_n$. Wu \cite{Wu97} expressed this tensor product explicitly:
\begin{equation} \label{eq: Wu's expr of tensor prod}
(X\tensor M)_n = \coker\left((M_n)_*\to \Ast_{x\in X_n} (M_n)_x \right).
\end{equation}
In that article, Wu computed the classifying space of $X\tensor M$ to be the following.

\begin{corollary}[Wu] \label{cor: Wu's result}
Let $M$ be a simplicial monoid and $X$ be a pointed simplicial set. There is a homotopy equivalence:
\[
B(X\tensor M) \simeq |X|\smash BM.
\]
\end{corollary}

\begin{proof}
Verify the universal property (Proposition \ref{prop: univ prop of red universal mon} below) to show that $X\tensor M$ is $U[X//M]$ where $X//M$ is the pointed simplicial action category associated to the trivial action of $M$ on $X$. Hence by definition of the $J^M[X]$-construction:
\begin{equation} \label{eq: tensor prod as univ mon}
X\tensor M \isom J^M[X]
\end{equation}

From Example \ref{eg: MxG example}, set $G$ to be the trivial simplicial group. The example states that $BJ^{M}[X]$ is the homotopy cofiber of $|X|\to |X\rtimes_{M} EM|$, where $M$ acts trivially. Combine this with \eqref{eq: tensor prod as univ mon} to obtain:
\begin{equation} \label{eq: B tensor prod as stunted}
B(X\tensor M) \simeq \frac{|X\rtimes_M EM|}{|X|}
\end{equation}

By definition of the smash product:
\begin{equation} \label{eq: temp 1}
|X|\smash BM=\frac{|X|\times BM}{|X|\vee BM}=\frac{(|X|\times BM)/(\pt\times BM)}{|X|}
\end{equation}
Since the action is trivial, there is a homotopy equivalence:
\begin{equation} \label{eq: temp 2}
|X\rtimes_M EM|=|X\times_M EM|/|\pt\times_M EM| \simeq |X|\times BM/(\pt\times BM).
\end{equation}
Compare \eqref{eq: temp 1} and \eqref{eq: temp 2} to obtain:
\[
|X|\smash BM\simeq \frac{|X\rtimes_M EM|}{|X|}.
\]
Compare with \eqref{eq: B tensor prod as stunted} to obtain the required homotopy equivalence.
\end{proof}

The constructions of James and Milnor are classical instances of the tensor product $X\tensor M$. James' construction \cite{James55} on a pointed simplicial set $X$ is the simplicial monoid $J[X]$ obtained by applying the reduced free monoid functor in each dimension. James' construction is $X$ tensored with the discrete monoid $\N$ of natural numbers under addition:
\begin{equation} \label{eq: James as X tensor N}
J[X]\isom X\tensor \N
\end{equation}
This can be shown directly with \eqref{eq: Wu's expr of tensor prod} or by verifying the universal property of \eqref{eq: adjoint def of tensor prod} or \eqref{eq: tensor prod as univ mon}.

\begin{corollary}
Let $X$ be a pointed simplicial set. Denote by $\Sigma$ the reduced suspension. There is a homotopy equivalence:
\[
BJ[X] \simeq \Sigma |X|
\]
\end{corollary}

\begin{proof}
Recall that the circle is the classifying space of $\N$:
\begin{equation} \label{eq: circle as BN}
B\N \simeq S^1
\end{equation}
This can be deduced from a general result of Dwyer and Kan \cite{DwyerKan80}. If $W$ is the free category on a quiver (a multi-diagraph with loops), then it has the same classifying space as its groupoid completion:
\[
BW \simeq B\bar{W}
\]
The homotopy equivalence \eqref{eq: circle as BN} follows since $B\bar{\N}\simeq B\Z= K(\Z,1)= S^1$.

Apply Corollary \ref{cor: Wu's result}:
\begin{align*}
BJ[X] = B(X\tensor \N) &\simeq B\N \smash |X| \\
&\simeq S^1 \smash |X| \quad {\text{(by \eqref{eq: circle as BN})}}\\
&\simeq \Sigma |X|
\end{align*}
\end{proof}

Similarly, Milnor's $F[K]$ construction \cite{Milnor72} on a pointed simplicial set $K$ is obtained by applying the reduced free group functor in each dimension. (Milnor's original notation is $FK$. To indicate that this construction is reduced, we take the liberty in adding the square brackets.) Milnor's construction is $K$ tensored with the discrete group $\Z$ of integers under addition:
\begin{equation} \label{eq: Milnor as K tensor Z}
F[K]\isom K\tensor \Z
\end{equation}

The standard computation of $|F[K]|$ uses a simplicial argument that involves Kan's construction (see page 294 of \cite{GoerssJardine99}). Our proof below, in comparison, is categorial in nature.

\begin{corollary}
Let $K$ be a pointed simplicial set. Denote by $\Sigma$ the reduced suspension. There are homotopy equivalences:
\begin{align*}
BF[K] &\simeq \Sigma |K| \\
|F[K]| &\simeq \Omega \Sigma |K|
\end{align*}
\end{corollary}

\begin{proof}
The classifying space of Milnor's construction can be deduced from Corollary \ref{cor: Wu's result}:
\begin{align}
BF[K] = B(\Z \tensor K) &\simeq B\Z \smash |K| \notag\\
&\simeq S^1 \smash |K| \notag\\
&\simeq \Sigma |K|
\end{align}

Since $F[K]$ is a simplicial group, we can apply \ref{para: loop nerve sgrp}:
\begin{align*}
|F[K]| &\simeq \Omega B F[K]  \\
&\simeq \Omega \Sigma |K|
\end{align*}
\end{proof}

\begin{example}
Let $X$ be the simplicial circle $S^1$ and $G$ be a discrete simplicial group. By Corollary \ref{cor: Wu's result},
\[
B(S^1\tensor G)\simeq |S^1| \smash BG \simeq \Sigma K(G,1)
\]
Since $S^1\tensor G$ is a simplicial group, \ref{para: loop nerve sgrp} implies that
\[
|S^1\tensor G| \simeq\Omega\Sigma K(G,1).
\]

The simplicial group $S^1\tensor G$ has been used to study the suspension
of the Eilenberg-Mac Lane spaces \cite{MutluPorter98,MikhailovWu10,Wu01}.
\end{example}

\end{document}